\documentclass[11pt]{amsart}
\usepackage{lmodern}
\usepackage[top=1.5in, left=1.2in, right=1.2in, bottom=1.5in]{geometry}
\usepackage{amssymb}
\usepackage{comment}
\usepackage{amsmath, amscd}
\usepackage{csquotes}
\usepackage{amsthm}
\usepackage[all]{xy}
\usepackage{color}
\usepackage{enumitem}
\usepackage{wrapfig}
\usepackage{graphicx}
\usepackage{pst-all}
\usepackage{hyperref}
\usepackage{tikz-cd}

\usepackage{palatino}

\newtheorem{theorem}{Theorem}[section]
\newtheorem*{theorem*}{Theorem}
\newtheorem{lemma}[theorem]{Lemma}

\newtheorem{proposition}[theorem]{Proposition}
\newtheorem{thmx}{Theorem}
 
\newtheorem{thmxx}{Theorem}

\theoremstyle{definition}

\newtheoremstyle{fact}
  {.5em}
  {.5em}
  {\itshape}
  {}
  {\itshape}
  {.}
  {.5em}
  {\thmname{#1}\thmnumber{ #2}\thmnote{ (#3)}}
\theoremstyle{fact}
\newtheorem{fact}{Fact}

\begin{document}

\title{Scaling functions for graph directed Markov systems}

\author{Daniel Ingebretson}

\begin{abstract}
We introduce the scaling function associated to a graph directed Markov system, and show that it is a H\"{o}lder continuous function of the dual symbolic Cantor set. 
With some natural separation and regularity conditions, each such system has a unique Cantor limit set in Euclidean space.
We prove that the scaling function is a complete invariant of $ C^{1+\alpha} $ conjugacy between limit sets.
We conclude by relating the scaling function to the pressure, and discussing several applications to the dimension theory of limit sets.
\end{abstract}

\maketitle 

\section{Introduction}

Graph directed Markov systems are general dynamical systems used to model many more specific systems, including conformal expanding repellers, iterated function systems, Kleinian groups, and continued fractions.
The general construction consists of a directed multigraph, a sequence of metric spaces indexed by the vertex set, and a sequence of functions indexed by the edge set.
The graph determines which functions may be composed; along any oriented path in the graph we have a corresponding composition.
If the functions are all injective contractions, a fixed point argument shows that the system has a unique invariant limit set, and there is a bijective coding map between the limit set and a symbolic space of admissible words.

To study the dynamics and geometry of limit sets, it is necessary to make some assumptions on such a system. 
In this work, we make the standard assumptions that the spaces are Euclidean, the maps are conformal, and the open set condition is satisfied.
To prove the main theorem however, we make the additional assumptions that the derivatives are bounded away from zero, and the open set condition is upgraded to a strong separation condition.
These imply that the limit set is a Cantor set in Euclidean space, the intersection over all $ n \in \mathbb{N} $ of nested sets coded by admissible words of length $ n $.

We consider two natural notions of equivalence between such limit Cantor sets.
We say two such sets are geometrically equivalent if their contraction rates are equal at arbitrarily fine scales.
They are differentiably equivalent of class $ C^{1+\alpha} $ if there is a map of the ambient Euclidean space taking one set into the other.
Each equivalence class is a differential structure on the Cantor set.

Motivated by the phenomenon of Feigenbaum universality, Sullivan \cite{Sul} studied differential structures on Cantor attractors of unimodal maps of the interval.
To find invariants, he introduced the ratio geometry and scaling function for these Cantor sets.
The dual Cantor set was originally defined to study smoothness properties of the shift map on the underlying symbolic space  \cite{Prz1}.
Sullivan claimed that the differential structure was completely determined by exponential geometric structure, and that the limiting scaling function is H\"{o}lder when the shift map is $ C^{1+\alpha} $.
Proofs of Sullivan's results appeared later \cite{Prz1}, \cite{Prz2}, and \cite{Bed2}, for homeomorphic embeddings of $ \{0,1\}^{\mathbb{N}} $ into the interval.
Related questions were also answered there, including conditions on the scaling function that imply higher smoothness of the shift map.

We begin by defining the ratio geometry sequence on the dual symbolic set for a graph directed Markov system.
Using distortion estimates, we show that this sequence limits to a well-defined H\"{o}lder scaling function.
Using the ratio geometry sequence we define geometric equivalence, and generalize Sullivan's theorem to limit sets of graph directed systems.

\begin{thmxx}
Two limit sets of conformal graph directed Markov systems are $ C^{1+\alpha} $ equivalent for some $ \alpha > 0 $, if and only if they are geometrically equivalent.
\end{thmxx}

In the second part of the paper, we study the pressure and its relation to the Hausdorff dimension of the limit set.
There is a natural generalization of Bowen's equation for graph directed Markov systems, which expresses the Hausdorff dimension of the limit set as the zero of the pressure function.
After recalling the construction of the pressure for graph directed systems, we show that it can be expressed entirely in terms of the scaling function on the dual Cantor set.

\begin{thmxx}
Let $ S $ be a conformal graph directed Markov system, with scaling function $ r $ and pressure $ P $. 
Then
$$
P(t) = \lim_{n \to \infty} \frac{1}{n} \log \sum_{\omega \in \widetilde{\Sigma}^n} \prod_{k=0}^{n-1} r( \ldots, \omega_n, \ldots, \omega_{k+1})^t.
$$
\end{thmxx}

In the case of similarity transformations, there is a classical dimension theory in terms of the similarity coefficients that goes back to Moran and Hutchinson (see \cite{Pes}).
If the maps in the graph directed system are similarities, the scaling function reduces to the similarity coefficients.
In this way, Theorem B and Bowen's equation can be viewed as a generalization of this theory to graph directed Markov systems.
In the conclusion of the paper, we illustrate this generalization by showing that in the cases of the infinitely generated similarities of \cite{MU1} and the graph directed constructions of \cite{MW}, Theorem B immediately reduces to the simple dimension formulas obtained there.

\section{Graph directed Markov systems} 

First we fix some notation. 
If $ A $ is a bounded metric space we denote its diameter by $ \text{diam}(A) $.
If $ x \in \mathbb{R}^d $, its Euclidean norm is denoted by $ |x| $, and if $ f : \mathbb{R}^d \rightarrow \mathbb{R}^d $ is $ C^1 $, the norm of its derivative at $ x \in \mathbb{R}^d $ is $ |Df(x)| $.
Finally, if $ f : X \rightarrow \mathbb{R}^d $ is a function on a compact metric space $ X $, we set $ \| f\| = \sup \{ |f(x)| : x \in X \} $, the usual uniform norm.

\subsection{Preliminaries: directed multigraphs}
The following notation is standard (see \cite{MU2}).
Consider a directed multigraph $ (V,E) $ with a finite vertex set $ V $ and a countable edge set $ E $. 
The functions $ i,t : E \rightarrow V $ are defined as follows: $ i(e) $ is the initial vertex of edge $ e $, and $ t(e) $ is its terminal vertex.
Let $ A : E \times E \rightarrow \{0,1\} $ be the edge incidence matrix of this graph, so that $ A_{e_1 e_2} = 1 $ if $ t(e_1) = i(e_2) $, and otherwise $ A_{e_1 e_2} = 0 $.
We denote the right-infinite words $ \omega = (\omega_1, \omega_2, \ldots ) $ on the alphabet $ E $ simply by $ E^{\infty} $.
The matrix $ A $ defines the set of infinite admissible words
$$
\Sigma_A = \{ \omega \in E^{\infty} : A_{\omega_j \omega_{j+1}} = 1 \text{ for all } j \geq 1 \},
$$
as well as each set of finite admissible words
$$
\Sigma_A^n = \{ \omega \in E^n : A_{\omega_j \omega_{j+1}} = 1 \text{ for all } 1 \leq j \leq n-1 \}
$$
of length $ n \geq 1 $.
Each word in $ \Sigma_A $ or $ \Sigma_A^n $ represents the symbolic dynamics of the orbit of an infinite or finite walk on the directed graph.
When the matrix $ A $ is fixed, we will refer to $ \Sigma_A $ and $ \Sigma_A^n $ simply by $ \Sigma $ and $ \Sigma^n $, respectively.
If $ \omega \in \Sigma^n $, we say that $ |\omega|=n $ is the \textit{word length} of $ \omega $.
If $ \omega = (\omega_1, \omega_2, \ldots) \in \Sigma $, we denote by
$$
\omega |_n = (\omega_1, \ldots, \omega_n) \in \Sigma^n $$
its right-truncation to length $ n $.

If $ \omega, \tau \in \Sigma $, we define the integer $ N = N(\omega, \tau) \geq 0 $ as the length of their longest common initial subword:
\begin{equation}
\label{N}
N(\omega, \tau) = \max \{ n : \omega_j = \tau_i \text{ for all } 1 \leq j \leq n \}.
\end{equation}
On $ E $ we take the discrete topology, the product topology on $ E^{\infty} $, and the subspace topology on $ \Sigma \subset E^{\infty} $.
If $ 0 < t < 1 $ is a constant, the topology induced by the metric $ \rho(\cdot, \cdot) = t^{N(\cdot, \cdot)} $ is compatible with this topology, and $ \Sigma $ is a Cantor set.

\subsection{Transpose graphs and the dual}
Now consider the transpose graph of $ (V,E) $, i.e. the graph with the same vertex and edge sets but directed by the transpose $ A^T $ of the incidence matrix, so that the direction of the edges are reversed.
Let $ E^{-\infty} $ be the space of left-infinite words $ \omega = (\ldots, \omega_2, \omega_1) $ on the alphabet $ E $.
Define
$$
\widetilde{\Sigma}_A = \{ \omega \in E^{-\infty} : A_{\omega_{j+1} \omega_j} = 1 \text{ for all } j \geq 1 \},
$$
and similarly define $ \widetilde{\Sigma}_A^n $.
Because we index the letters $ \omega_j $ from right to left, and use the transpose $ A^T $, note that $ \widetilde{\Sigma}_A^n = \Sigma_{A}^n $ for any $ n \geq 1 $, but $ \widetilde{\Sigma}_A $ is distinct from $ \Sigma_{A} $.
As before, when the matrix $ A $ is fixed, we will suppress the reference and simply write $ \widetilde{\Sigma} $ and $ \widetilde{\Sigma}^n $.
If $ \omega \in \widetilde{\Sigma}^n $, its word length $ |\omega| $ is $ n $, and if $ \omega = (\ldots, \omega_2, \omega_1) \in \widetilde{\Sigma} $, the notation for left-truncation is
$$
\omega |_n = (\omega_n, \ldots, \omega_1) \in \widetilde{\Sigma}^n.
$$

We can metrize $ \widetilde{\Sigma} $ in an analogous way to $ \Sigma $; for $ \omega, \tau \in \widetilde{\Sigma} $ let $ N = N(\omega,\tau) $ be the length of their longest common initial subword, beginning from the right. 
Because the index increases from right to left in our notation of dual words, the definition of $ N $ in Equation \ref{N} is valid here as well.
Then for any $ 0 < t < 1 $, the metric $ \widetilde{\rho}(\cdot, \cdot) = t^{N(\cdot, \cdot)} $ induces the product topology on $ \widetilde{\Sigma} $, in which $ \widetilde{\Sigma} $ is a Cantor set, called the \textit{dual Cantor set} to $ \Sigma $.

\subsection{Graph directed Markov systems}
Fix a constant $ 0 < \lambda < 1 $.
Let $ \{ X_v \}_{v \in V} $ be a collection of non-empty compact metric spaces indexed by the vertex set $ V $, and assume that for each edge $ e \in E $ we have an injective contraction map $ \phi_e : X_{t(e)} \rightarrow X_{i(e)} $ with Lipschitz constant $ \leq \lambda $.
The set 
$$ 
S = \{ \phi_e : X_{t(e)} \rightarrow X_{i(e)} \}_{e \in E} 
$$ 
is called a \textit{graph directed Markov system} or GDMS.

\subsubsection{Sets coded by $ \Sigma $}
Consider a finite word $ (\omega_1, \ldots, \omega_n) \in \Sigma^n $.
For each $ 1 \leq j \leq n-1 $, we have $ A_{\omega_j \omega_{j+1}} = 1 $, so that $ t(\omega_j) = i(\omega_{j+1}) $ and the directed graph contains the following subgraph:
$$
X_{i(\omega_j)} \longrightarrow \left( X_{t(\omega_j)} = X_{i(\omega_{j+1})} \right) \longrightarrow X_{t(\omega_{j+1})},
$$
and thus the composition 
$$ 
\phi_{\omega_j} \circ \phi_{\omega_{j+1}} : X_{t(\omega_{j+1})} \rightarrow X_{i(\omega_j)} 
$$
is well-defined.
We can now define the iterated composition map 
$$ 
\phi_{\omega_1, \ldots, \omega_n} =  \phi_{\omega_1} \circ \cdots \circ \phi_{\omega_n} : X_{t(\omega_n)} \rightarrow X_{i(\omega_1)}.
$$
and denote its image by
$$ 
\Delta_{\omega_1, \ldots, \omega_n} = \phi_{\omega_1, \ldots, \omega_n}\left( X_{t(\omega_n)} \right). 
$$
For each $ \omega \in \Sigma $, we have a nesting property $ \Delta_{\omega |_n} \supset \Delta_{\omega |_{n+1}} $, so that $ \bigcap_{n \geq 1} \Delta_{\omega |_n} \neq \emptyset $.
Furthermore, because the Lipschitz constant of each map $ \phi_e $ is $ \leq \lambda $, then
\begin{equation}
\label{contract}
\text{diam}(\Delta_{\omega |_n}) \leq \lambda^n \text{diam}(X_{t(\omega_n)}) \leq \lambda^n \max \{ \text{diam}(X_v) : v \in V \}.
\end{equation}
Since $ 0 < \lambda < 1 $, we know that $ \bigcap_{n \geq 1} \Delta_{\omega |_n} $ is a singleton, which defines a coding map 
$$ 
\pi : \Sigma \rightarrow \bigcup_{v \in V} X_v, \qquad \pi(\omega) = \bigcap_{n \geq 1} \Delta_{\omega |_n}.
$$
The set $ J_S = \pi(\Sigma) $ is called the \textit{limit set} of the GDMS $ S $.
When the GDMS $ S $ is fixed, we simply write $ J = J_S $.

\subsubsection{Sets coded by $ \widetilde{\Sigma} $}
Now consider a finite dual word $ (\omega_n, \ldots, \omega_1) \in \widetilde{\Sigma}^n $.
For each $ 1 \leq j \leq n-1 $, we have $ A_{\omega_{j+1} \omega_j} = 1 $, so that $ t(\omega_{j+1}) = i(\omega_j) $ and the transpose graph contains the subgraph:
$$
X_{i(\omega_{j+1})} \longrightarrow \left( X_{t(\omega_{j+1})} = X_{i(\omega_j)} \right) \longrightarrow X_{t(\omega_j)},
$$
and the composition
$$
\phi_{\omega_{j+1}} \circ \phi_{\omega_j} : X_{t(\omega_j)} \rightarrow X_{i(\omega_{j+1})} 
$$
is well-defined.
This allows us to define the composition map
$$ 
\phi_{\omega_n, \ldots, \omega_1} =  \phi_{\omega_n} \circ \cdots \circ \phi_{\omega_1} : X_{t(\omega_1)} \rightarrow X_{i(\omega_n)},
$$
with image
$$
\Delta_{\omega_n, \ldots, \omega_1} = \phi_{\omega_n, \ldots, \omega_1}\left( X_{t(\omega_1)} \right).
$$
Of course, we have an analogous inequality to Equation \ref{contract} for left-truncated dual words.

\subsection{Conformal systems}
\label{conf}
We now list several more assumptions that will be key in later studies of limit sets.

\begin{enumerate}[label=(\alph*),leftmargin=2.5\parindent]
\item The spaces $ X_v $ are compact, convex, and lie in a common subspace of Euclidean space $ \mathbb{R}^d $ for some $ d \geq 1 $.
\vspace{0.1cm}
\item \textit{(Open set condition)} For all $ e_1, e_2 \in E $,
$$
\phi_{e_1}\left( \text{int}(X_{t(e_1)}) \right) \cap \phi_{e_2} \left( \text{int}(X_{t(e_2)}) \right) = \emptyset.
$$
\vspace{0.1cm}
\item For each $ v \in V $ there is an open neighborhood $ W_v $ of $ X_v $, so that each map $ \phi_e : X_{t(e)} \rightarrow X_{i(e)} $ extends to a $ C^1 $ conformal diffeomorphism $ W_{t(e)} \rightarrow W_{i(e)} $.
\vspace{0.1cm}
\item There exists constants $ C, \alpha > 0 $ such that 
$$
\Big| |D\phi_e(x)| - |D\phi_e(y)| \Big| \leq C \|(D\phi_e)^{-1}\|^{-1} |x-y|^{\alpha}
$$
for all $ e \in E $ and $ x,y \in X_{t(e)} $.
\vspace{0.1cm}
\end{enumerate}

If a GDMS satisfies (a) -- (d), we call it a \textit{conformal} GDMS or CGDMS.
For the remainder of this work, we will only be concerned with limit sets of CGDMS.

The convexity assumption in (a) can be replaced by a weaker cone condition. 
When $ d \geq 2 $, condition (d) is implied by conditions (a) and (c), as a consequence of the Koebe distortion theorem.
When the alphabet $ E $ is finite, it is not necessary to impose condition (d), as long as the maps $ \phi_e $ are of class $ C^{1+\alpha} $. 
For details, see \cite{MU1} and \cite{MU2}.

Later we will need additional assumptions from the following list, which we will specify when needed.

\begin{enumerate}[leftmargin=2.5\parindent]
\item[(b')] \textit{(Strong separation condition)} There exists a constant $ a > 0 $ such that for all $ e_1, e_2 \in E $, $ \Delta_{e_1} $ and $ \Delta_{e_2} $ are separated by at least $ a $.
\vspace{0.1cm}
\item[(e)] \textit{(Finite primitivity)} For some $ n \geq 1 $ there exists a finite set of words $ \Lambda \subset \Sigma^n $ such that for all $ e_1, e_2 \in E $ there exists $ (\omega_1, \ldots, \omega_n) \in \Lambda $ such that $ (e_1, \omega_1, \ldots, \omega_n, e_2) \in \Sigma^{n+2} $.
\vspace{0.1cm}
\item[(f)] \textit{(exponential geometry)} The following constant is nonzero:
$$
\inf_{e \in E} \inf_{x \in X_{t(e)}} | D\phi_e(x) | = \lambda^- > 0.
$$
\vspace{0.1cm}
\end{enumerate}

For the remainder of this section, we will discuss several consequences of these assumptions, which will be used extensively in later sections.
We will not require finite primitivity (e) until we study the pressure in Section \ref{pres}.

If $ S $ satisfies the strong separation property (b') (or possibly weaker; $ S $ is pointwise finite) we have the following characterization of the limit set $ J $:
\begin{equation}
\label{J}
J = \pi(\Sigma) = \bigcup_{\omega \in \Sigma} \bigcap_{n \geq 1} \Delta_{\omega |_n} = \bigcap_{n \geq 1} \bigcup_{\omega \in \Sigma^n} \Delta_{\omega},
\end{equation}
and $ J $ is totally disconnected. \
It is perfect because of the uniform contraction property from Equation \ref{contract}. 
Since each $ X_v $ is compact, $ J $ is a Cantor set in $ \mathbb{R}^d $.
In fact, $ \pi $ is a homeomorphism.
In the absence of a separation condition, however, $ \pi $ fails to be injective.

For any word $ \omega \in \Sigma $ we have the nesting condition $ \Delta_{\omega |_n} \supset \Delta_{\omega |_{n+1}} $.
Then for any $ n \geq 1 $, the collections $ \bigcup_{\omega \in \Sigma^n} \Delta_{\omega} $ contain $ J $ and as $ n $ increases, these collections are successively better approximations of $ J $.

There is no analogous nesting condition for dual words $ \omega \in \widetilde{\Sigma} $; in fact, $ \Delta_{\omega_n, \ldots, \omega_1} $ is disjoint from $ \Delta_{\omega_{n+1}, \ldots, \omega_1} $ unless $ \omega_n = \omega_{n+1} $.
Thus the intersection $ \bigcap_{n \geq 1} \Delta_{\omega |_n} $ is usually empty for dual words $ \omega \in \widetilde{\Sigma} $, with the exception of the admissible constant words $ (\ldots, e, e) $ for some $ e \in E $.
So the coding map $ \pi : \Sigma \rightarrow \bigcup_{v \in V} X_v $ has no extension to the dual $ \widetilde{\Sigma} $.
However, because $ \Sigma^n = \widetilde{\Sigma}^n $ for each $ n \geq 1 $, we may write Equation \ref{J} as
$$
J = \bigcap_{n \geq 1} \bigcup_{\omega \in \widetilde{\Sigma}^n} \Delta_{\omega},
$$
and as $ n $ increases, the collection of dual sets $ \bigcup_{\omega \in \widetilde{\Sigma}^n} \Delta_{\omega} $ also provide successively better approximations of $ J $.
From Equation \ref{contract}, we know that $ \text{diam}(\Delta_{\omega |_n}) \rightarrow 0 $ and the convergence is \textit{at least} exponential.
If we assume (f), then
$$
\text{diam}(\Delta_{\omega |_n}) \geq (\lambda^-)^n.
$$
This implies that the convergence is \textit{precisely} exponential, hence the term `exponential geometry.'
With this assumption, we can upgrade the separation condition (b') as follows.

\begin{lemma}
\label{sep}
Let $ S $ be a CGDMS satisfying the strong separation and exponential geometry assumptions.
Let $ \omega \neq \tau \in \Sigma $, and let $ N = N(\omega, \tau) $ be defined in Equation \ref{N}.
For any $ n > N $, there exists $ a > 0 $ such that the sets $ \Delta_{\omega |_n} $ and $ \Delta_{\tau |_n} $ are separated by at least $ a (\lambda^-)^N $.
\end{lemma}

\begin{proof}
By definition, $ \omega_{N+1} \neq \tau_{N+1} $, so by condition (b'), there exists a constant $ a > 0 $ such that $ \Delta_{\omega_{N+1}} $ and $ \Delta_{\tau_{N+1}} $ are separated by at least $ a $.
By the nesting property, this implies for any $ n > N $ that $ \Delta_{\omega_{N+1}, \ldots, \omega_n} $ and $ \Delta_{\tau_{N+1}, \ldots, \tau_n} $ are also separated by at least $ a $.
From the mean value theorem and condition (f),
$$
|\phi_{\omega_1, \ldots, \omega_N}(x) - \phi_{\omega_1, \ldots, \omega_N}(y)| \geq (\lambda^-)^N |x-y|.
$$
Because $ \omega_i = \tau_i $ for all $ 1 \leq i \leq N $, the claim follows.
\end{proof}

\section{Ratio geometry and the scaling function}
In this section we will introduce the ratio geometry on the dual, and use this to define the scaling function.
Convergence of the scaling function will follow from the following important bounded distortion property, which we phrase in terms of the dual $ \widetilde{\Sigma} $.

\begin{proposition}[\textit{Bounded distortion}]
\label{bddist}
Let $ S = \{ \phi_e \}_{e \in E} $ be a CGDMS.
For any $ n,m \geq 1 $, $ \omega = (\ldots, \omega_2, \omega_1) \in \widetilde{\Sigma} $, and $ x,y \in \Delta_{\omega_n, \ldots, \omega_1} $, there exists a constant $ K > 0 $ such that
$$
e^{-K\lambda^{\alpha n}} \leq \frac{|D \phi_{\omega_{n+m}, \ldots, \omega_{n+1}}(x)|}{|D \phi_{\omega_{n+m}, \ldots, \omega_{n+1}}(y)|} \leq e^{K\lambda^{\alpha n}}.
$$
\end{proposition}

\begin{proof}
Fix $ \omega = (\ldots, \omega_2, \omega_1) \in \widetilde{\Sigma} $, and $ x,y \in \Delta_{\omega_n, \ldots, \omega_1} $.
Define the sequence $ x_k $ by $ x_1 = x $ and $ x_k = \phi_{\omega_{n+j-1}, \ldots, \omega_{n+1}}(x) $ for $ 2 \leq k \leq m $, and similarly define $ y_k $ in terms of $ y $.
Because $ x,y \in \Delta_{\omega_n, \ldots, \omega_1} $, note that $ x_j, y_j \in \Delta_{\omega_{n+j-1}, \ldots, \omega_1} $.
Using this, and assumption (d) in Section \ref{conf}, we have
\begin{align*}
\left| \log \frac{|D \phi_{\omega_{n+m}, \ldots, \omega_{n+1}}(x)|}{|D \phi_{\omega_{n+m}, \ldots, \omega_{n+1}}(y)|} \right| &= \left| \sum_{j=1}^m \log \left( 1+\frac{|D\phi_{\omega_{n+j}}(x_j)| - |D\phi_{\omega_{n+j}}(y_j)|}{|D\phi_{\omega_{n+j}}(x_j)|} \right) \right| \\
&\leq \sum_{j=1}^m \big\|(D\phi_{\omega_{n+j}})^{-1} \big\| \: \Big| |D\phi_{\omega_{n+j}}(x_j)| - |D\phi_{\omega_{n+j}}(y_j)| \Big| \\
&\leq \sum_{j=1}^m C |x_j - y_j|^{\alpha} \\
&\leq C \sum_{j=1}^m \lambda^{\alpha(n+j-1)} \leq \frac{C \lambda^{\alpha n}}{1-\lambda^{\alpha}}.
\end{align*}
Setting $ \displaystyle K = \frac{C}{1-\lambda^{\alpha}} $ concludes the proof.
\end{proof}

\subsection{Ratio geometry}
\label{ratgeo}
Fix a CGDMS $ S $, determining the sets $ \Delta_{\omega |_n} $ for each dual word $ \omega = (\ldots, \omega_2, \omega_1) \in \widetilde{\Sigma} $.
Set 
$$
r_1(\omega_1) = \frac{\text{diam}\left(\phi_{\omega_1}(X_{t(\omega_1)})\right)}{\text{diam}(X_{t(\omega_1)})},
$$
and for each $ n \geq 2 $ set
$$
r_n(\omega |_n) = \frac{\text{diam}(\Delta_{\omega_n, \ldots, \omega_1})}{\text{diam}(\Delta_{\omega_n, \ldots, \omega_2})}.
$$
From the strict containment $ \Delta_{\omega_n, \ldots, \omega_1} \subsetneq \Delta_{\omega_n, \ldots, \omega_2} $ and assumption (e) from Section \ref{conf}, we have $ 0 < r_n(\omega |_n) < 1 $ for all $ n \geq 1 $.
The sequence of functions $ r_n : \widetilde{\Sigma}^n \rightarrow (0,1) $ is called the \textit{ratio geometry} of the GDMS $ S $.
The bounded distortion property implies the following distortion estimate for the ratio geometry sequence.
\begin{proposition}
\label{dualrat}
There exists a constant $ K > 0 $ such that for all $ \omega \in \widetilde{\Sigma} $ and $ n,m \geq 1 $,
$$
e^{-K \lambda^{n \alpha}} \leq \frac{r_{n+m}(\omega |_{n+m})}{r_n(\omega |_n)} \leq e^{K \lambda^{n \alpha}}.
$$
\end{proposition}

\begin{proof}
By the mean value theorem there exist $ x \in \Delta_{\omega_n, \ldots, \omega_1} $ and $ y \in \Delta_{\omega_n, \ldots, \omega_2} $ such that
\begin{align*}
\text{diam}(\Delta_{\omega_{n+m}, \ldots, \omega_1}) &= |D\phi_{\omega_{n+m}, \ldots, \omega_{n+1}}(x)| \: \text{diam}(\Delta_{\omega_n, \ldots, \omega_1}), \text{ and } \\
\text{diam}(\Delta_{\omega_{n+m}, \ldots, \omega_2}) &= |D\phi_{\omega_{n+m}, \ldots, \omega_{n+1}}(y)|\: \text{diam}(\Delta_{\omega_n, \ldots, \omega_2}).
\end{align*}
Combining these equations yields
$$
\frac{r_{n+m}(\omega |_{n+m})}{r_n(\omega |_n)} = \frac{|D \phi_{\omega_{n+m}, \ldots, \omega_{n+1}}(x)|}{|D \phi_{\omega_{n+m}, \ldots, \omega_{n+1}}(y)|}.
$$
Because $ \Delta_{\omega_n, \ldots, \omega_1} \subset \Delta_{\omega_n, \ldots, \omega_2} $, the desired inequality now follows from Proposition \ref{bddist}.
\end{proof}

\subsection{The scaling function}
As $ n \rightarrow \infty $, the ratio geometry sequence $  r_n(\omega |_n) $ measures the contraction rate at arbitrarily small scales.
Along dual words $ \omega \in \widetilde{\Sigma} $, this rate approaches a constant:
\begin{equation}
\label{scaleq}
r(\omega) = \lim_{n \to \infty} r_n(\omega |_n).
\end{equation}
The function $ r : \widetilde{\Sigma} \rightarrow (0,1) $ is called the \textit{scaling function} on the dual $ \widetilde{\Sigma} $.

\begin{proposition}
\label{scale}
For each $ \omega \in \widetilde{\Sigma} $ the limit in Equation \ref{scaleq} exists, and the convergence is exponential in $ n $.
\end{proposition}

For the proof (and for later proofs) we will require an auxiliary lemma.

\begin{lemma}
\label{aux}
For $ C, A, \delta > 0 $ and $ 0 < t < 1 $, the sequences $ \log(1+C e^{-n \delta}) $, $ \log(1-C e^{-n \delta}) $, and $ A t^n $ are all asymptotically equivalent; i.e. given $ C, \delta > 0 $ there exist $ A > 0 $ and $ 0 < t < 1 $ such that $ \log(1+C e^{-n \delta}) \leq At^n $ for all $ n \geq 1 $, and there are identical statements comparing all pairs of these three sequences.
\end{lemma}

\begin{proof}
The proof follows easily from the Taylor expansion of $ \log(1 \pm x) $ about $ x=0 $.
\end{proof}

\begin{proof}[Proof of Proposition \ref{scale}]
By Proposition \ref{dualrat}, 
$$
\left| \log \left( \frac{r_{n+m}(\omega |_{n+m})}{r_n (\omega |_n)} \right) \right| \leq K \lambda^{n \alpha}
$$
for all $ n,m \geq 1 $.
Because $ \lambda < 1 $, this shows that the sequence $ \log r_n(\omega |_n) $ is Cauchy.
Because $ r_n(\omega |_n) $ is bounded away from zero for each $ \omega \in \widetilde{\Sigma} $, the limit $ r(\omega) $ exists.

To see that the convergence is exponential, take $ m \to \infty $ in Proposition \ref{dualrat}, which yields
$$
e^{-K \lambda^{n \alpha}} \leq \frac{r_n(\omega |_n)}{r(\omega)} \leq e^{K \lambda^{n \alpha}}.
$$
By Lemma \ref{aux}, there exist constants $ C, \delta > 0 $ such that $ K \lambda^{n \alpha} \leq \log (1+C e^{-n \delta}) $, and $ -K \lambda^{n \alpha} \geq \log (1-C e^{-n \delta}) $.
Setting $ A = C r(\omega) $, we obtain
$$
| r_n(\omega |_n) - r(\omega) | \leq A e^{-n \delta}.
$$
\end{proof}

Recall that the dual space $ \widetilde{\Sigma} $ is metrized by $ \rho(\cdot, \cdot) = t^{N(\cdot, \cdot)} $ where $ 0 < t < 1 $ is any constant and $ N = N(\omega, \tau) $ is defined in Equation \ref{N}. 
As a function between metric spaces, the scaling function satisfies the following property.

\begin{proposition}
The scaling function $ r : \widetilde{\Sigma} \rightarrow (0,1) $ is H\"{o}lder continuous.
\end{proposition}

\begin{proof}
Let $ \omega \neq \tau \in \widetilde{\Sigma} $, so that $ 0 \leq N(\omega, \tau) < \infty $.
By Proposition \ref{dualrat}, there exists a constant $ K > 0 $ such that for all $ n,m \geq 1 $ we have
$$
e^{-K \lambda^{n\alpha}} \leq \cfrac{ \cfrac{r_{n+m}(\omega |_{n+m})}{r_{n+m}(\tau |_{n+m})} }{ \cfrac{r_n(\omega |_n)}{r_n(\tau |_n)} } \leq e^{K \lambda^{n\alpha}}
$$
Now set $ n = N(\omega, \tau) $, so the denominator of the above fraction equals 1. 
Taking $ m \to \infty $ yields
$$
\left| \log \frac{r(\omega)}{r(\tau)} \right| \leq K \lambda^{N\alpha}.
$$
Because $ 0 < \lambda < 1 $, the metric $ \widetilde{\rho}(\cdot, \cdot) = \lambda^{N(\cdot, \cdot)} $ generates the topology on $ \widetilde{\Sigma} $.
Then the above inequality reads
$$
\left| \log \frac{r(\omega)}{r(\tau)} \right| \leq K \rho(\omega, \tau)^{\alpha},
$$
so $ \log r $ is H\"{o}lder continuous, and thus so is $ r $.
\end{proof}

\section{Differential and geometric equivalence}
In this section we study differential and geometric equivalence for limit sets of CGDMS.
Fix a directed graph $ (V,E) $ with incidence matrix $ A $, a family of spaces $ \{ X_v \}_{v \in V} $, and consider two CGDMS $ S = \{ \phi_e : X_{t(e)} \rightarrow X_{i(e)} \}_{e \in E} $ and $ T = \{ \psi_e : X_{t(e)} \rightarrow X_{i(e)} \}_{e \in E} $ defined by this directed graph.
These CGDMS have limit sets $ J_S $ and $ J_T $, with coding maps $ \pi_S : \Sigma \rightarrow \cup_v X_v $ and $ \pi_T : \Sigma \rightarrow \cup_v X_v $.
It is important here that $ \Sigma = \Sigma_A $ and $ A $ is the same for $ J_S $ and $ J_T $; while the maps $ \phi_e $ and $ \psi_e $ may be very different, the underlying directed graphs are equal.
We will not consider equivalence between CGDMS with different directed graphs.

We say that $ J_S $ and $ J_T $ are \textit{$ C^{1+\alpha} $-equivalent} if there exists a $ C^{1+\alpha} $ diffeomorphism $ \Phi : \mathbb{R}^d \rightarrow \mathbb{R}^d $ such that the following diagram commutes.
\vspace{0.1cm}
\begin{center}
\begin{tikzcd}
   & & \bigcup_{v \in V} X_v \arrow[dd, "\Phi |_{J_T} "] \\
    \Sigma_A \arrow[urr, "\pi_T"] \arrow[drr, "\pi_S"'] & &  \\ 
   & & \bigcup_{v \in V} X_v
\end{tikzcd}
\vspace{0.1cm}
\end{center}
 
Recall the ratio geometry sequence $ r_n : \widetilde{\Sigma} \rightarrow (0,1) $ on the dual.
We now extend this to a sequence $ r_n : \Sigma^n \rightarrow (0,1) $ in the following way. 
$$
r_n(\omega |_n) = \frac{\text{diam}(\Delta_{\omega_1,\ldots,\omega_n})}{\text{diam}(\Delta_{\omega_1,\ldots,\omega_{n-1}})}.
$$
Let $ r_n(\omega |_n) $ and $ s_n(\omega |_n) $ be the ratio geometry sequences of $ J_S $ and $ J_T $ respectively, defined on $ \Sigma^n $.
We say $ J_S $ and $ J_T $ have \textit{equivalent geometries} if for all $ (\omega_1, \omega_2, \ldots) \in \Sigma $,
$$
\frac{r_n(\omega |_n)}{s_n(\omega |_n)} \rightarrow 1
$$
as $ n \rightarrow \infty $. 
If the convergence is exponential in $ n $, we say that the geometries of $ J_S $ and $ J_T $ are \textit{exponentially equivalent}.
Notice that if $ J_S $ and $ J_T $ have equivalent geometries, then their scaling functions are equal.
We can now prove the following theorem from the introduction.

\begin{thmx}
Two limit sets of CGDMS satisfying the exponential geometry and strong separation conditions are $C^{1+\alpha}$-equivalent for some $ \alpha > 0 $ if and only if their geometries are exponentially equivalent.
\end{thmx}

\begin{proof}
We begin by fixing some notation. 
Let $ S = \{ \phi_e \}_{e \in E} $ and $ T = \{ \psi_e \}_{e \in E} $ be two CGDMS with limit sets
$$
J_S = \bigcap_{n \geq 1} \bigcup_{\omega \in \Sigma^n} \Delta_{\omega}, \qquad J_T = \bigcap_{n \geq 1} \bigcup_{\omega \in \Sigma^n} \Theta_{\omega}.
$$
We assume that the maps $ \phi_e $ all have Lipschitz constant $ \leq \lambda $, and $ \psi_e $ all have Lipschitz constant $ \leq \eta $ for some $ 0 < \lambda, \eta < 1 $.
From the exponential geometry assumption in Section \ref{conf}(f), the constant $ \lambda^- $ is nonzero, and $ \eta^- $ is defined similarly:
$$
\eta^- = \inf_{e \in E} \inf_{x \in X_{t(e)}} |D\psi_e(x)| > 0.
$$

First, assume that $ J_S $ and $ J_T $ are $C^{1+\alpha}$-equivalent. 
Then there exists $ \Phi : \mathbb{R}^d \rightarrow \mathbb{R}^d $ of class $ C^{1+\alpha} $ such that $ \Phi(J_T) = J_S $, so $ \Phi(\Theta_{\omega |_n}) = \Delta_{\omega |_n} $ for each $ \omega \in \Sigma $ and $ n \geq 1 $.
By the mean value theorem, there exist $ x \in \Theta_{\omega_1, \ldots, \omega_n} $ and $ y \in \Theta_{\omega_1, \ldots, \omega_{n-1}} $ such that
\begin{align*}
\text{diam}(\Delta_{\omega_1, \ldots, \omega_n}) &= |D\Phi(x)| \: \text{diam}(\Theta_{\omega_1, \ldots, \omega_n}), \text{ and } \\
\text{diam}(\Delta_{\omega_1, \ldots, \omega_{n-1}}) &= |D\Phi(y)| \: \text{diam}(\Theta_{\omega_1, \ldots, \omega_{n-1}}).
\end{align*}
Dividing these two equations yields
\begin{equation}
\label{rnsn}
\frac{r_n(\omega |_n)}{s_n(\omega |_n)} = \frac{|D\Phi(x)|}{|D\Phi(y)|}.
\end{equation}
Because $ \Phi $ is $ C^{1+\alpha} $ with derivative bounded away from zero, we may assume that $ \log |D\Phi| $ is H\"{o}lder, i.e. there exist $ C, \alpha > 0 $ such that for all $ x, y \in J_T $,
$$
\left|\log \frac{|D\Phi(x)|}{|D\Phi(y)|}\right| \leq C |x-y|^{\alpha}.
$$
In particular, this holds for $ x,y \in \Theta_{\omega_1, \ldots, \omega_{n-1}} $ from Equation \ref{rnsn}, so for this choice,
$$
\left|\log \frac{|D\Phi(x)|}{|D\Phi(y)|}\right| \leq C \lambda^{n \alpha}.
$$
Substituting this into Equation \ref{rnsn} we obtain
$$
e^{-C \lambda^{n \alpha}} \leq \frac{r_n(\omega |_n)}{s_n(\omega |_n)} \leq e^{C \lambda^{n \alpha}}.
$$
As $ n \rightarrow \infty $, the above ratio geometry quotient approaches $ 1 $ and thus $ J_S $ and $ J_T $ have equivalent geometries.
By Lemma \ref{aux}, there exists $ A, \delta >0 $ such that $ C \lambda^{n \alpha} \leq \log(1+Ae^{-n \delta}) $, and $ \log(1-A e^{-n \delta}) \leq -C \lambda^{n \alpha} $, which yields
$$
\left| \frac{r_n(\omega |_n)}{s_n(\omega |_n)} - 1 \right| \leq A e^{-n \delta},
$$
so the convergence is exponential.

Conversely, assume that $ J_S $ and $ J_T $ have exponentially equivalent geometries. 
Define $ \Phi : J_T \rightarrow J_S $ by $ \Phi = \pi_S \circ \pi_T^{-1} $.
Then $ \Phi(\Theta_{\omega_1, \ldots, \omega_n}) = \Delta_{\omega_1, \ldots, \omega_n} $ for each finite word $ (\omega_1, \ldots, \omega_n) \in \Sigma^n $.
We wish to extend $ \Phi $ to a $ C^{1+\alpha} $ map $ \mathbb{R}^d \rightarrow \mathbb{R}^d $.
To do this, we will use the following vector-valued version of the $ C^{1+\alpha} $ Whitney extension theorem (see \cite{Gal}).

\begin{fact}
\label{whit}
Let $ X \subset \mathbb{R}^d $ be a closed set, $ f : X \rightarrow \mathbb{R}^d $ a $ C^{1+\alpha} $ map, and $ g : X \rightarrow \text{GL}_d(X, \mathbb{R}^d) $ a function satisfying
$$
\lim_{|x-y| \rightarrow 0} \frac{\|g(x) - g(y)\|}{|x-y|^{\alpha}} = 0
$$
where $ \| \cdot \| $ is any of the equivalent matrix norms on $ \text{GL}_d(\mathbb{R}^d) $.
Then $ f $ extends to a $ C^{1+\alpha} $ map $ \mathbb{R}^d \rightarrow \mathbb{R}^d $, and $ Df = g $.
\end{fact}

Let $ x \in J_T $, so there exists a unique $ \omega \in \Sigma $ such that $ x = \pi_T(\omega) $.
In the usual coordinate system on $ \mathbb{R}^d $ we write $ x = (x_1, \ldots, x_d) $, and denote the component functions of $ \Phi $ by $ \{ \Phi_i(x_1, \ldots, x_d) \}_{i=1}^d $.
We define the partial derivatives $ \partial \Phi_i / \partial x_j $ at the point $ x $ as follows.
Let $ x_{\omega |_n}^{\pm} \in \mathbb{R}^d $ be the points for which
\begin{equation}
\label{xomega}
\text{diam}(\Theta_{\omega |_n}) = \left| x_{\omega |_n}^+ - x_{\omega |_n}^- \right|.
\end{equation}
For $ j=1, \ldots, d $, we denote the components of these points by $ x_{\omega |_n, j}^{\pm} $, and for each $ i=1, \ldots, d $, let $ y_{\omega |_n, i}^{\pm} $ be the points for which
\begin{equation}
\label{yomega}
\text{diam}(\Phi_i (\Theta_{\omega |_n}) ) = |y_{\omega |_n, i}^+ - y_{\omega |_n, i}^-|.
\end{equation}
In terms of these, define
\begin{equation}
\label{partial}
\frac{\partial \Phi_i}{\partial x_j}(x) = \lim_{n \to \infty} \frac{|y_{\omega |_n, i}^+ - y_{\omega |_n, i}^- |}{|x_{\omega |_n, j}^+ - x_{\omega |_n, j}^-|}
\end{equation}
for each $ j $ such that $ x_{\omega |_n, j}^+ \neq x_{\omega |_n, j}^- $; if $ x_{\omega |_n, j}^+ = x_{\omega |_n, j}^- $, we set $ \partial \Phi_i / \partial x_j (x) = 0 $.
Let $ D \Phi(x) \in GL_n(\mathbb{R}^d) $ be the Jacobian matrix of partial derivatives.

First, we will show that $ \Phi : J_T \rightarrow \mathbb{R}^d $, with derivative $ D \Phi $ so defined, is of class $ C^{1+\alpha} $.
By the mean value theorem,
$$
|D\Phi(x)| = \lim_{n \to \infty} \frac{\text{diam}(\Delta_{\omega |_n})}{\text{diam}(\Theta_{\omega |_n})}.
$$
Now let $ x \neq y \in J_T $, so that $ x = \pi_T(\omega) $ and $ y = \pi_T(\tau) $ for some $ \omega \neq \tau \in \Sigma $, say.
In terms of these, define
$$
a_n = \cfrac{\cfrac{\text{diam}(\Delta_{\omega |_n})}{\text{diam}(\Theta_{\omega |_n})}}{\cfrac{\text{diam}(\Delta_{\tau |_n})}{\text{diam}(\Theta_{\tau |_n})}},
$$
so that $ \displaystyle \lim_{n \to \infty} a_n = \frac{|D\Phi(x)|}{|D\Phi(y)|} $.
The sequence $ a_n $ and the ratio geometry quotients satisfy the following cross-ratio:
\begin{equation}
\label{crossrat}
a_{n-1} \: \frac{r_n(\omega |_n)}{s_n(\omega |_n)} = \frac{r_n(\tau |_n)}{s_n(\tau |_n)} \: a_n
\end{equation}
Because we assumed the ratio geometries are exponentially equivalent, there exist constants $ A, \delta > 0 $ such that for all $ n \geq 0 $, the quotients $ r_n(\omega |_n) / s_n(\omega |_n) $ and $ r_n(\tau |_n) / s_n(\tau |_n) $ both lie in the interval $ \displaystyle [ 1-Ae^{-n \delta}, 1+Ae^{-n \delta} ] $.
By Lemma \ref{aux}, there exist constants $ 0 < t < 1 $ and $ C > 0 $ such these two quotients both lie in the interval $ \displaystyle [ e^{-(C/2)t^n}, e^{(C/2)t^n} ] $.
Dividing these two inequalities yields
$$
e^{-C t^n} \leq \cfrac{ \cfrac{r_n(\tau |_n)}{s_n(\tau |_n)} }{ \cfrac{r_n(\omega |_n)}{s_n(\omega |_n)} } \leq e^{C t^n},
$$
which by Equation \ref{crossrat} implies
$$
e^{-C t^n} \leq \frac{a_n}{a_{n-1}} \leq e^{C t^n}.
$$
Via an inductive and geometric series argument (and renaming the constant $ C $), we can improve this to
\begin{equation}
\label{anm}
e^{-C t^n} \leq \frac{a_{n+m}}{a_n} \leq e^{C t^n}
\end{equation}
for any $ n,m \geq 1 $.

Now fix $ N = N(\omega, \tau) $, the length of the longest common initial subword of $ \omega $ and $ \tau $ as defined in Equation \ref{N}.
Notice that $ a_N = 1 $.
Substituting $ N = n $ and taking $ m \to \infty $ in Equation \ref{anm}, we obtain
\begin{equation}
\label{eqa}
\left| \log \frac{|D\Phi(x)|}{|D\Phi(y)|} \right| \leq C t^N
\end{equation}
Take $ n > N $. 
By the strong separation assumption and Lemma \ref{sep}, we know that $ x \in \Theta_{\omega |_n} $ and $ y \in \Theta_{\tau |_n} $ are separated by at least $ a (\lambda^-)^N $.
Take $ 0 < \alpha < 1 $ so that $ t \leq (\lambda^-)^{\alpha} $.
Then Equation \ref{eqa} reads
$$
\left| \log \frac{|D\Phi(x)|}{|D\Phi(y)|} \right| \leq C a^{-\alpha} |x-y|^{\alpha},
$$ 
so $ \log |D\Phi| $ is H\"{o}lder continuous, and so is $ |D\Phi| $.

It remains to show that $ D\Phi $ satisfies the second hypothesis of Fact \ref{whit}, namely that
$$
\lim_{|x-y| \rightarrow 0} \frac{\|D\Phi(x) - D\Phi(y)\|}{|x-y|^{\alpha}} = 0.
$$
We begin with two inequalities, which follow from Equations \ref{xomega} and \ref{yomega}.
\begin{align*}
\frac{1}{\sqrt{d}} \: \text{diam}(\Theta_{\omega |_n}) &\leq \max_{1 \leq j \leq d} |x_{\omega |_n, j}^+ - x_{\omega |_n, j}^-| \leq \text{diam}(\Theta_{\omega |_n}), \\
\frac{1}{\sqrt{d}} \: \text{diam}(\Delta_{\omega |_n}) &\leq \max_{1 \leq i \leq d} |y_{\omega |_n, i}^+ - y_{\omega |_n, i}^-| \leq \text{diam}(\Delta_{\omega |_n}).
\end{align*}
Dividing these two equations, taking $ n \to \infty $ and using our geometric equivalence assumption and the definition in Equation \ref{partial}, we obtain
$$
\frac{1}{\sqrt{d}} \leq \max_{1 \leq i,j \leq d} \frac{\partial \Phi_i}{\partial x_j} \leq \sqrt{d}.
$$
We have an analogous equation for $ y = \pi_T(\tau) $. 
Taking $ \| \cdot \| $ as the usual max norm on $ \text{GL}_d(\mathbb{R}^d) $, this implies
$$
\| D\Phi(x) - D\Phi(y) \| \leq \frac{d-1}{\sqrt{d}},
$$
which concludes the proof.
\end{proof}

\section{Pressure and the scaling function}
\label{pres}
In this section we will introduce the pressure, and relate it to the scaling function, proving the second theorem from the introduction.
In the conclusion of this section, we will mention some applications to the Hausdorff dimension of limit sets of similarity mappings.
We begin by collecting some known facts about the pressure.
If $ S = \{ \phi_e : X_{t(e)} \rightarrow X_{i(e)} \}_{e \in E} $ is a CGDMS determining sets $ \Delta_{\omega |_n} $ for each $ \omega \in \Sigma $ and $ n \geq 1 $, we define the topological pressure $ P: [0,\infty) \rightarrow \mathbb{R} $ by
\begin{equation}
\label{Pt}
P(t) = \lim_{n \to \infty} \frac{1}{n} \log \sum_{\omega \in \Sigma^n} \| D\phi_{\omega} \|^t.
\end{equation}
The limit exists because the sequence $ a_n = \sum_{\omega \in \Sigma^n} \|D\phi_{\omega}\|^t $ is subadditive for each $ t \geq 0 $. 
Let $ F = \{ t \geq 0 : P(t) < \infty \} $ be the set of finiteness of $ P $, and $ \theta = \inf F $.
If we assume finite primitivity (assumption (e) from Section \ref{conf}), then $ P $ has the following properties.

\begin{fact}[Proposition 4.2.8 from \cite{MU2}]
The topological pressure $ P(t) $ of a finitely primitive CGDMS is non-increasing on $ [0,\infty) $, strictly decreasing to $ -\infty $ on $ [\theta, \infty) $, and is convex and continuous on $ F $.
\end{fact}

We are now in a position to prove the second theorem from the introduction, which we restate below.

\begin{thmx}
Let $ S $ be a finitely primitive CGDMS, with scaling function $ r : \Sigma \rightarrow (0,1) $ and pressure $ P $. 
Then
$$
P(t) = \lim_{n \to \infty} \frac{1}{n} \log \sum_{\omega \in \widetilde{\Sigma}^n} \prod_{k=0}^{n-1} r( \ldots, \omega_n, \ldots, \omega_{k+1})^t.
$$
\end{thmx}

\begin{proof}
Let $ p_n(t) = \sum_{\omega \in \Sigma^n} \| D\phi_{\omega} \|^t $.
Because $ \Sigma^n = \widetilde{\Sigma}^n $ for any $ n \geq 1 $, we have $ p_n(t) = \sum_{\omega \in \widetilde{\Sigma}^n} \| D\phi_{\omega} \|^t $.
We will require the following fact.

\begin{fact}[Equations 4.20 and 4.23 from \cite{MU2}]
There exists a constant $ K>0 $ such that for all finite words $ \omega \in \Sigma^n $,
$$
K^{-1} \leq \frac{\text{diam}(\Delta_{\omega})}{\|D \phi_{\omega} \|} \leq K.
$$
\end{fact}

From this, we obtain
\begin{equation}
\label{eqb}
K^{-1} \leq \frac{p_n(t)}{\sum_{\omega \in \widetilde{\Sigma}^n} \text{diam}(\Delta_{\omega})^t} \leq K.
\end{equation}
We can relate this to the ratio geometry sequence on the dual, from Equation \ref{ratgeo}:
$$
\sum_{\omega \in \widetilde{\Sigma}^n} \text{diam}(\Delta_{\omega})^t = \sum_{\omega \in \widetilde{\Sigma}^n} \prod_{k=0}^{n-1} r_{n-k}(\omega_n, \ldots, \omega_{k+1})^t.
$$
By Proposition \ref{dualrat}, there exists a constant $ A > 0 $ such that
$$
\left| \log \frac{r_{n+m-k}(\omega_{n+m}, \ldots, \omega_{k+1}}{r_{n-k}(\omega_n, \ldots, \omega_{k+1})} \right| \leq A \lambda^{(n-k) \alpha}.
$$
Taking $ m \rightarrow \infty $ and substituting into Equation \ref{eqb} gives the bounds
$$
K^{-1} p_n^-(t) \leq p_n(t) \leq K p_n^+(t),
$$
where
$$
p_n^{\pm}(t) = \sum_{\omega \in \widetilde{\Sigma}^n} \prod_{k=0}^{n-1} r(\ldots, \omega_n, \ldots, \omega_{k+1})^t e^{\pm t A\lambda^{(n-k)\alpha}}.
$$
By a geometric series argument, $ \prod_{k=0}^{n-1} e^{t A \lambda^{(n-k)\alpha}} \leq B^t $ for some constant $ B > 0 $ independant of $ n $.
The bounds then improve to
$$
K^{-1} B^{-t} \leq \frac{p_n(t)}{\sum_{\omega \in \widetilde{\Sigma}^n} \prod_{k=0}^{n-1} r(\ldots, \omega_n, \ldots, \omega_{k+1})^t} \leq K B^t.
$$
Taking log, dividing by $ n $, and taking $ n \to \infty $ yields the claim.
\end{proof}

\subsection{Applications to dimension theory of limit sets}
The pressure of a CGDMS is related to the dimension of its limit set by the following generalization of Bowen's formula.

\begin{fact}[Theorem 4.2.13 from \cite{MU2}]
\label{bow} Let $ S $ be a finitely primitive CGDMS, $ J $ its limit set, and $ \text{dim}_H $ the Hausdorff dimension.
Then
$$
\text{dim}_H(J) = \inf \{ t \geq 0 : P(t) < 0\},
$$
and if $ P(t_{\ast}) = 0 $, then $ t_{\ast} $ is the unique zero of $ P $ and $ t_{\ast} = \text{dim}_H(J) $.
\end{fact}

In light of this result and Theorem B, the Hausdorff dimension of the limit set $ J $ can be computed entirely in terms of the scaling function on the dual $ \widetilde{\Sigma} $.
And in the cases where the maps are similarities (i.e. their derivatives are constants), we now show that Theorem B reduces to several well-known formulas for the Hausdorff dimension of Cantor sets.

\subsubsection{Iterated similarity maps}
We take the vertex set a single point $ X_v = X \subset \mathbb{R}^d $, and the maps $ \phi_e : X \rightarrow X $ are
$$
\phi_e (x) = \lambda_e A \circ I + b,
$$
where $ \lambda_e > 0 $ is a positive scalar, $ A $ is a linear isometry in $ \mathbb{R}^d $, $ I $ is either an inversion with respect to a fixed sphere $ S^{d-1} $ of a given radius and center, and $ b \in \mathbb{R}^d $ is any vector.
The constant $ \lambda_e $ is called the \textit{similarity coefficient} of the map $ \phi_e $.
For any point $ x \in X_{t(e)} $ we have $ |D\phi_e(x)| = \lambda_e $, and thus $ \text{diam} (\phi_e(X_{t(e)})) = \lambda_e \text{diam}(X_{t(e)}) $.
Then for any dual word $ (\ldots, \omega_2, \omega_1) \in \widetilde{\Sigma} $, the scaling function depends only on the ``first" letter:
$$
r(\ldots, \omega_2, \omega_1) = \lambda_{\omega_1}.
$$
Substituting this expression into the expression for the pressure derived in Theorem B yields
$$
P(t) = \lim_{n \to \infty} \frac{1}{n} \log \sum_{\omega \in \Sigma^n} \lambda_{\omega_1}^t \cdots \lambda_{\omega_n}^t = \lim_{n \to \infty} \frac{1}{n} \log \left( \sum_{e \in E} \lambda_e^t \right)^n = \log \sum_{e \in E} \lambda_e^t.
$$
Thus by Fact \ref{bow}, the Hausdorff dimension of the limit set $ J $ is $ \inf \{ t \geq 0 : \sum_{e \in E} \lambda_e^t < 1 \} $, which is Corollary 3.17 from \cite{MU1}.
And if the alphabet $ E $ is finite, say $ E = \{1, \ldots, p\} $, then the Hausdorff dimension is $ t_{\ast} $ satisfies $ \sum_{i=1}^p \lambda_i^{t_{\ast}} = 1 $, which is Moran's fundamental formula \cite{Pes}.

\subsubsection{Graph directed constructions}
The limit sets of similarity maps composed along a directed graph has a well-developed dimension theory initiated in \cite{MW} (see also \cite{Pes}).
We summarize their work below, and show that the dimension formula for the limit set obtained there is a special case of our Theorem B, together with Fact \ref{bow}.

Let $ J_1, \ldots, J_p $ be compact nonoverlapping subsets of $ \mathbb{R}^d $ with nonempty interior.
Consider a connected directed graph $ G $ on the vertex set $ \{1, \ldots, p\} $, and for each $ (i,j) \in G $, a similarity map $ T_{i,j} : J_j \rightarrow J_i $ with similarity coefficient $ \lambda_{i,j} $.
Such a system is called a \textit{graph directed construction}.
With some natural separation conditions, this system can be shown to have an invariant Cantor set $ K = \bigcup_i K_i $, and each $ K_i $ is self-similar in the sense that 
$$
K_i = \bigcup_{(i,j) \in G} T_{i,j}(K_j).
$$
We are primarily interested in the Hausdorff dimension of $ K $.

Such a directed graph is a special case of the multigraphs we have considered in this paper, but in this case, there are no distinct edges with the same initial and terminal vertices.
Considered as a multigraph $ (V, E) $, we have $ V = \{ 1, \ldots, p \} $ and $ E = \{ (i,j) \in G \} $ so $ E \subset V \times V $ in this case.
Because the maps are similarities, the scaling function on the dual only depends on the ``first'' letter (as above), so
$$
r(\ldots, (i_3, i_2), (i_2, i_1)) = \lambda_{i_2, i_1}
$$
for each dual word.
Substituting into the pressure formula in Theorem B yields
$$
P(t) = \lim_{n \to \infty} \frac{1}{n} \log \sum_{\omega \in \Sigma^n} \prod_{i=1}^p \lambda_{\omega_i, \omega_{i+1}}^t.
$$
Let $ A_t $ be the $ p \times p $ matrix with entries $ \lambda_{i,j}^t $, the ``construction matrix'' of the graph-directed construction in the terminology of \cite{MW}.
This matrix is irreducible because the graph is connected, and thus its spectral radius $ \Phi(t) $ is positive by the Perron-Frobenius theorem.
Let $ \| \cdot \|_1 $ be the $ l_1 $ norm on $ \mathbb{R}^d $, and $ u = (1, \ldots, 1) \in \mathbb{R}^d $.
By induction on $ n \geq 1 $,
$$
\| A_t^n u \|_1 = \sum_{\omega \in \Sigma^n} \prod_{i=1}^p \lambda_{\omega_i, \omega_{i+1}}^t.
$$
Taking the entrywise matrix norm $ \| \cdot \| = \| \cdot u \|_1 $ we thus have that $ P(t) = \log \Phi(t) $ by Gelfand's theorem.
Because the alphabet is finite, by Theorem B and Fact \ref{bow} we have that the Hausdorff dimension $ t_{\ast} $ of the limit set $ K $ is the unique solution to $ \Phi(t_{\ast}) = 1 $, which is the fundamental result of \cite{MW}.

\end{document}